\documentclass{article}
\usepackage[english]{babel}
\usepackage[usenames]{color}
\usepackage{amsmath,amsfonts,amsthm,amssymb,amscd}
 \usepackage{verbatim}
      \usepackage{showkeys} 

\renewcommand{\Re}{\mathop{\rm Re}\nolimits}

\newcommand{\p}{\partial}
\newcommand{\e}{\varepsilon}
\newcommand{\vp}{\varphi}

\newcommand{\strela}{\rightharpoonup}
\newcommand{\ri}{{\rightarrow}}
\newcommand{\C}{{\mathbb C}}
\newcommand{\R}{{\mathbb R}}
\newcommand{\Z}{{\mathbb Z}}

\newcommand{\pP}{{\mathbb P}}

\newcommand{\E}{{\mathbb E}}

\newcommand{\N}{{\mathbb N}}
\newcommand{\la}{\lambda}

\newcommand{\ty}{\infty}

\newcommand{\om}{\omega}

\newcommand{\CK}{C_0(K)}

\newcommand{\BBB}{{\boldsymbol{B}}}

\newcommand{\DD}{{\cal D}}

\newcommand{\FF}{{\cal F}}

\newcommand{\HH}{{\cal H}}

\newcommand{\LL}{{\cal L}}

\newcommand{\PP}{{\cal P}}

\newcommand{\UU}{{\cal U}}
\newcommand{\VV}{{\cal V}}

\newcommand{\tata}{\mbox{\boldmath$\tau$}}

\newcommand{\lag}{\langle}
\newcommand{\rag}{\rangle}

\newcommand{\dd}{{\textup d}}

\newcommand{\PPPP}{{\mathfrak P}}

\newcommand{\BBBBB}{{\mathcal B}}

\newcommand{\Lip}{\mathop{\rm Lip}\nolimits}

\theoremstyle{plain}
\newtheorem{theorem}{Theorem}[section]
\newtheorem{lemma}[theorem]{Lemma}
\newtheorem{proposition}[theorem]{Proposition}
\newtheorem{corollary}[theorem]{Corollary}
\newtheorem{definition}[theorem]{Definition}

\theoremstyle{remark}
\newtheorem{remark}[theorem]{Remark}

\newcommand{\de}{\delta}
\numberwithin{equation}{section}
\begin{document}

\title{Stochastic CGL equations without linear dispersion  in any space dimension 
}
\date{\today}

\author{Sergei
Kuksin\footnote{CNRS and CMLS, Ecole Polytechnique,
91128, Palaiseau, France, e-mail:
kuksin@math.polytechnique.fr}, Vahagn Nersesyan\footnote{Laboratoire de
Math\'ematiques, UMR CNRS 8100, Universit\'e de
Versailles-Saint-Quentin-en-Yvelines, F-78035 Versailles,
France,
e-mail: Vahagn.Nersesyan@math.uvsq.fr}}
 \maketitle

{\small\textbf{Abstract.}    
We consider the stochastic CGL equation
$$
\dot u- \nu\Delta u+(i+a) |u|^2u =\eta(t,x),\;\;\; \text {dim} \,x=n,
$$
where $\nu>0$ and  $a\ge 0$, 
in a cube (or in a smooth bounded domain) with 
 Dirichlet boundary condition. The force $\eta$ is white in time, regular in $x$ and 
non-degenerate. We study this equation in the space of continuous complex functions
$u(x)$, and prove that for any $n$ it defines there a
unique mixing Markov process. So for a large class of functionals $f(u(\cdot))$ and for any solution 
 $u(t,x)$,  the averaged 
observable $\E f(u(t,\cdot))$ converges to a quantity, independent from the initial data
$u(0,x)$, and equal to the integral of $f(u)$  against the unique stationary measure of the equation.
}
  \tableofcontents

\section{Introduction}\label{S:intr}

We study the stochastic CGL equation 
\begin{equation}\label{1}
\dot u- \nu\Delta u+(i+a) |u|^2u =\eta(t,x), \quad  \text {dim} \,x=n,
\end{equation} 
where $n$ is any, $\nu>0$, $a\ge0$ and the random force $\eta$ is white in time and 
regular in $x$. All our results and constructions are uniform in $a$ from bounded intervals 
$[0,C]$, $C\ge0$.  Since for $a>0$ the
equation  possesses extra properties due to the nonlinear dissipation (it is ``stabler"), 
then below  we restrict ourselves to the case $a=0$; see discussion in 
 Section~\ref{S:5}. 
This equation 
is the Hamiltonian system $\ \dot u+i|u|^2u=0$, damped by the 
viscous term $\nu\Delta u$ and driven by the random force $\eta$. So it makes a model for the stochastic 
Navier-Stokes system, which may be regarded as 
 a damped--driven Euler equation (which is  a Hamiltonian system, homogeneous of degree  two). 
 In this work we are not concerned with the interesting turbulence-limit
$\nu\to0$ (see \cite{K97, K99} for some related results) and, again to simplify notation, choose 
$\nu=1$. That is, we consider the equation 
\begin{align}
\dot u-\Delta u+i |u|^2u =\eta(t,x).\label{E:1}
\end{align}   
For the space-domain we take the cube $K=[0,\pi]^n$ with the Dirichlet boundary conditions,
which we regard as the odd periodic boundary conditions
 \begin{equation*}
u(t,\ldots, x_j,\ldots)=u(t,\ldots, x_j+2\pi,\ldots)=-u(t,\ldots, -x_j,\ldots) \quad\forall\,j.
\end{equation*} 
Our results remain true for \eqref{E:1} in a smooth bounded domain with the  Dirichlet boundary conditions,
see Section~\ref{S:5}.

 The force $\eta(t,x)$ is a random field of the form
\begin{align}\label{E:2*}
\eta(t,x)=\frac{\p}{\p t} \zeta(t,x),\quad \zeta(t,x)=\sum_{s\in\N^n} b_s\beta_s(t) \varphi_s(x).
\end{align}Here $b_s$ are real numbers such that 
\begin{align*} 
B_*:=\sum_{s\in\N^n} |b_s|<\ty,  
\end{align*} 
  $\beta_s= \beta_s^++i\beta_s^-$, where $\beta_s^{\pm}$ are standard independent
    (real-valued)   Brownian motions,  defined on a complete probability 
space $(\Omega,\FF,\pP)$ with a filtration $\{\FF_t; t\ge0\}$.\footnote{
The filtered probability space  $(\Omega,\FF,  \{\FF_t\}, \pP)$, as well as all other filtered  probability space,
used in this work, are assumed to satisfy the usual condition, see \cite{KS91}.}

The  set of real functions
 $\{\varphi_s(x), s  \in \N^n\}$ is the $L^2$-normalised system of eigenfunctions  of the  Laplacian,
 \begin{equation*}
 \begin{split}
\varphi_s(x)=\left( {2}/{\pi}\right)^{{n}/{2}}\sin s_1 x_1\cdot\ldots\cdot\sin s_n x_n,\qquad
(-\Delta)\varphi_s=\alpha_s\varphi_s,\quad \alpha_s=|s|^2.
\end{split}
\end{equation*} 
Our work continues the research \cite{K99} and makes use of its method which exploits essentially the well known
fact that the deterministic  equation \eqref{E:1}${}_{\eta=0}$ implies for the real function $|u(t,x)|$ a 
parabolic inequality with the maximum principle. Denote by $H^m$ the Sobolev space of order $m$,
formed  by complex  odd periodic functions and given the norm 
\begin{equation}\label{norm}
\|u\|_m=\|(-\Delta)^{m/2}u\|,
\end{equation}
where $\|\cdot\|$ is the $L^2$-norm on the cube $K$. In Section~\ref{S:21} we repeat some construction 
from \cite{K99} and state its main result,  which says  that if \begin{align}
u(0,x)=u_0(x), \,\,\,\, \label{E:5}
\end{align}    
where $u_0\in H^m$, $m>n/2$, and 
$$
B_m:=\sum b_s^2|s|^{2m}<\infty,
$$
then \eqref{E:1}, \eqref{E:5} has a unique strong solution $u(t)\in H^m$. Moreover, for any $T\ge0$ the
random variable $X_T=\sup_{T\le t\le T+1}|u(t)|^2_\infty$ satisfies the estimates
\begin{equation}\label{7}
\E X_T^q\le C_q \qquad \forall\,q\ge0, 
\end{equation} 
where $C_q$ depends only on $|u_0|_\infty$ and $B_*$. Analysis of the constants $C_q$,
made in Section~\ref{S:gn}, implies that suitable exponential moments of the variables $X_T$ 
are finite:
\begin{equation}\label{8}
\E e^{cX_T} \le C'=C'(B_*,|u_0|_\infty),
\end{equation} 
where $c>0$ depends only on  $B_*$. 

Denote by $C_0(K)$ the space of continuous complex 
functions on $K$, vanishing at $\partial K$.  In Section~\ref{S:3} we consider  the initial-value problem 
 \eqref{E:1}, \eqref{E:5}, assuming only  that $B_*<\infty$ and $u_0\in C_0(K)$. Approximating it by the regular 
 problems as above and using that the constants in \eqref{7}, \eqref{8} depend only on  $B_*$  and $|u_0|_\infty$,
 we prove
 \smallskip
 
 \noindent
 {\bf Theorem A.} Let   $B_*<\infty$ and $u_0\in C_0(K)$. Then the problem \eqref{E:1}, \eqref{E:5} has a
 unique strong solution $u(t,x)$ which  almost surely belongs to the space 
 $
 \  C([0,\infty), C_0(K))\cap L^2_{loc}([0,\infty),H^1)\,.
 $
 The solutions $u$  define in the space  $C_0(K)$ a Fellerian Markov process. 
 \smallskip
 
 Consider the quantities $J^t=\int_0^t|u(\tau)|^2_\infty d\tau-Kt$, where $K$ is a suitable constant, depending 
 only on $B_*$. Based on \eqref{8}, we prove in Lemma~\ref{P:3} that the random variable $\sup_{t\ge0}J^t$
 has exponentially bounded tails. Since the 
 non-autonomous term in the linearised equation \eqref{E:1}  is quadratic in $u,\bar u$, then the method to treat the
 2d stochastic Navier-Stokes system, based on the Foias-Prodi estimate and the Girsanov theorem (see \cite{KS11}
 for discussion and references to the original works) allows us to prove in Section~\ref{S:4} 
 \smallskip
 
 \noindent
{\bf  (stability)} There is a constant $L\ge1$ and 
 two sequences $\{T_m\ge0, m\ge1\}$ and $\{\e_m>0, m\ge1\}$, $\e_m\to0$ as $m\to\infty$, 
such that if for any $m\ge1$  we have solutions
  $u(t), u'(t)$  of \eqref{E:1}, satisfying 
  $$
  u(0),u'(0)\in B_m=\{u\in C_0(K): \|u\| \le 1/m,\   |u|_{L_\infty}\le L\},
  $$
     then for
 each $t\ge T_m$ we have
$ \| \DD(u(t))- \DD(u'(t))\|^*_\LL 
\le \e_m$. 
 Here $\| \mu-\nu \|^*_\LL$ is the dual-Lipschitz distance between Borelian measures  $\mu$ and $\nu$ 
on the space $H^0$ (see below Notation).
\smallskip

We also verify   in Section~\ref{S:4}  that 

 \noindent
{\bf  (recurrence)}  For each $m\ge 1$ and for any $u_0,u'_0\in C_0(K)$, the hitting time $\inf\{t\ge0: u(t)\in B_m,
u'(t)\in B_m\}$, where $u(t)$ and $u'(t)$ are two independent solutions of \eqref{E:1} such that $u(0)=u_0$ and
$u'(0)=u'_0$,    is almost surely finite. 
\smallskip

These two properties allow us to apply to eq.  \eqref{E:1} an abstract theorem from \cite{KS11}
\footnote{That result was introduced in \cite{ ShiII}, based on 
  ideas, developed in  \cite{KSC} to establish mixing for the stochastic
 2D NSE. It applies to various  nonlinear stochastic PDEs, including the complicated 
  CGL equation \eqref{1} where $a=0$ and $\nu$ is complex number with a positive real part, see 
\cite{ShiII}. }
which  implies the second main result of this work: 
\smallskip
 
 \noindent
 {\bf Theorem B.}  There is an integer $N = N(B_*,\nu)\ge1 $   such that if
 $ b_s\ne0$  for $|s|\le N$, 
   then the Markov process, constructed in  Theorem A, is mixing. That is, it has a unique stationary 
 measure $\mu$, and every solution $u(t)$ converges to $\mu$ in distribution
 \smallskip
 
 This theorem implies that for a large class of continuous functionals $f$ on $C_0(K)$ we have the 
 convergence
 $$
 \E f(u(t))\to \int f(v)\,\mu(dv)\quad\text{as}\quad t\to\infty, 
 $$
 where $u(t)$ is any solution of \eqref{E:1}.     See 
  Corollary \ref{C:43}. 
 
 In Section \ref{S:5} we explain that our results also apply to equations \eqref{1}, considered in smooth bounded 
 domains in $\R^n$ with Dirichlet boundary conditions;  that Theorem~A generalises to equations
 \begin{equation}\label{9}
\dot u- \nu\Delta u+(i+a)g_r( |u|^2)u =\eta(t,x), 
\end{equation} 
where $g_r(t)$ is a smooth function, equal to $t^r, r\ge0$, for $t\ge1$, and Theorem B generalises to eq.~\eqref{9} 
with $0\le r\le1$. 
\medskip

Similar results for the CGL equations \eqref{9},  where $\eta$ is a kick force, are obtained  in  \cite{KS11} 
 without the restriction that the nonlinearity
is cubic, and for the case when $\eta$   is the  derivative of a compound Poisson
process -- in \cite{N1}.
Our technique does not apply to equations \eqref{9} with complex $\nu$.  To prove analogies of Theorems~A, B
for such equations, strong restrictions should be imposed on $n$ and $r$. See \cite{ Hai01b, Oda08} for
equations with Re$\,\nu>0$ and $ a>0$, and see \cite{ShiII}  for the  case Re$\,\nu>0$ and $ a=0$. We
also mention the work \cite{DOd} which treats interesting class of one-dimensional 
equations \eqref{1} with  complex $\nu$ such that  Re$\,\nu=0$ and $ a=0$, damped by the term $\alpha u$ in the l.h.s. of the 
equation.

\bigskip
\textbf{Notation.}
 By $H$ we denote the $L^2$-space of odd $2\pi$-periodic complex functions with the scalar product 
  $\lag u, v\rag\!:=\Re\int_K u(x)\bar v(x)\dd x$ and the norm $\|u\|^2:=\lag u, u\rag$;
 by  $H^m(K)$, $m\ge 0$ --  the Sobolev space of  odd $2\pi$-periodic complex functions of 
 order $m$, endowed with the homogeneous norm  \eqref{norm}
 (so $H^0(K)=H$ and $\|\cdot\|_0=\|\cdot\|$). 
By $C_0(Q)$ we denote the space of continuous complex functions on a closed domain 
$Q$ which vanish at the boundary $\partial Q$ (note that the space $C_0(K)$ is 
formed by restrictions to $K$ of continuous odd periodic functions).

  For a Banach space $X$ we denote:
\\$C_b(X)$ --  the space of real-valued bounded continuous functions on $X$;
\\$\LL(X)$ -- the space of bounded Lipschitz 
 functions $f$  on $X$, given the norm
$$
\|f\|_{\LL}:=|f|_\infty+\Lip(f)< \infty, \quad \Lip(f):=\sup_{u\neq v} 
{|f(u)-f(v)|}  \, {\|u-v\|}^{-1};
$$ 
 $\BBBBB(X)$  -- the $\sigma$-algebra of Borel subsets of $X$;
\\ $\PP(X)$ -- the set of probability measures on $(X,\BBBBB(X))$;\\
$B_X(d), d>0$ -- the open ball  in $X$ of radius $d$,
centered at the origin.\\
For $\mu\in\PP(X)$ and  $f\in C_b(X)$ we denote
$
(f,\mu)=(\mu,f)
=\int_Xf(u)\mu(\dd u).
$
If $\mu_1,\mu_2\in\PP(X)$, we set
\begin{equation*}\begin{split}
&\|\mu_1-\mu_2\|_\LL^*=\sup\{|(f,\mu_1)-(f,\mu_2)|: f\in\LL(X),
\|f\|_\LL\le1\},\\
&\|\mu_1-\mu_2\|_{var}=\sup\{|\mu_1(\Gamma)-\mu_2(\Gamma)|:\Gamma\in\BBBBB(X)\}.
\end{split}
\end{equation*} 
The arrow $\strela$ indicates the weak convergence of measures in $\PP(X)$. 
 It is well known that 
$\mu_n\strela\mu$ if and only if $\|\mu_n-\mu\|_\LL^*\to0$, and that 
$\|\mu_1-\mu_2\|_\LL^*\le 2 \|\mu_1-\mu_2\|_{var}$. 
 
The distribution of a random variable $\xi$ is denoted by
$\DD(\xi)$.  For complex numbers $z_1,z_2$ we denote $z_1\cdot z_2=$Re$\,z_1\bar z_2$; so
$z\cdot d\beta_s=($Re$\,z)d\beta^+_s+($Im$\,z)d\beta^-_s(t)$. 
We denote by $C, C_k$ unessential positive constants.

 \section{Stochastic CGL equation} \label{S:2}
 
\subsection{Strong and weak  solutions.
}  \label{S:21}

Let the filtered probability space  $(\Omega,\FF,  \{\FF_t\}, \pP)$ be as in Introduction.
We use the standard definitions of strong and weak solutions for  stochastic PDEs (e.g., see
\cite{KS91}):
\begin{definition}\label{D:1}
Let $0<T<\infty$. A random process  $u(t) = u(t, x), t\in [0,T]$ in $\CK$ defined on a probability space  
 $(\Omega,\FF,\pP)$ is called a       strong solution of  (\ref{E:1}),  (\ref{E:5})  if the following three
conditions hold:
 \begin{enumerate}
\item[(i)] the process  $u(t)$ is adapted to the filtration $\FF_t$;
\item[(ii)]  its trajectories  $u(t)$   a.s.  belong to the space 
$$
\HH([0,T]):=  C([0,T], \CK )\cap L^2([0,T],H^1);
$$
\item[(iii)] for every  $t\in[0,T]$ a.s. we have 
$$
 u(t) = u_0 +\int_0^t (\Delta u-i |u|^2u)\dd s+ \zeta(t),
$$
where both sides are regarded as elements of $H^{-1}$.
\end{enumerate}
\end{definition}

If (i)-(iii) hold for every $T<\infty$, then $u(t)$ is called a strong solution for $t\in\R_+=[0,\infty)$.

A continuous adapted process $u(t)\in C_0(K)$ and a Wiener process $\zeta'(t)\in H$, defined in some
filtered probability space, are called a weak solution of \eqref{E:1} if $\DD(\zeta')=\DD(\zeta)$ and (ii),~(iii) of 
Definition~\ref{D:1} hold with $\zeta$ replaced by $\zeta'$. 
 
We recall that for $l\ge0$ we denote
$B_l:=\sum_{s\in\N^n} |b_s|^2 |s|^{2l}\le\ty.
$ Note that $B_0\le B_*^2<\ty$. Let us fix any
$$
m>{n}/ {2}.
$$
  Problem (\ref{E:1}),  (\ref{E:5}) with $u_0\in H^m$   and $B_m<+\ty$ was considered in \cite{K99}. 
Choosing $\de=1$ in \cite{K99},   we state the main result of that work as follows:
\begin{theorem}\label{T:1}
Assume that $u_0\in H^m$ and $B_m<+\ty$. Then (\ref{E:1}),  (\ref{E:5}) has a unique strong 
solution $u$ which is  in   $\HH([0,\infty))$         
 a.s.,  and for any $t\ge0, q\ge1$ satisfies the estimates
\begin{align}
\E\sup_{s\in [t,t+1]} |u(s)|^q_\ty&\le C_q, \label{E:6} \\
\E \|u(t)\|_m^q& \le C_{q,m},  \nonumber
\end{align}where $C_q$ is a constant depending on $ |u_0|_\ty$, while $C_{q,m}$ also depends on $\|u_0\|_m$ and $B_m$.
\end{theorem} In this theorem and everywhere below the constants  depend on $n$ and $B_*$. We do
 not indicate this dependence. 
\begin{remark}It was assumed in \cite{K99} that $n\le 3$. This assumption is not needed for the proof.
 The force $\eta(t,x)$ in \cite{K99} has the form $\eta(t,x)\dot \beta(t)$, where $\beta$ is the  standard
 Brownian motion and $\eta(t,x)$ is a random field, continuous and bounded uniformly in $(t,x)$, smooth in $x$  and 
  progressively measurable. The proof without any change applies to forces  of the form (\ref{E:2*}). \end{remark}

Our next goal is to get more estimates for solutions $u(t,x)$. 
Applying It\^o's formula to $\|u\|^2$, where $u(t)=\sum u_s(t)\varphi_s(x)$ 
is a solution constructed in Theorem~\ref{T:1}, we find that
$$
\|u(t)\|^2=\|u_0\|^2+\int_0^t(-2\|u(\tau)\|_1^2+2B_0)\dd \tau+
2\sum_{s\in \N^n} b_s\int_0^t   u_s(\tau)
\cdot\dd \beta_s (\tau).
$$Taking the expectation, we get for any $t\ge0$
\begin{align}\label{E:31}
\E \|u(t)\|^2+2\E \int_0^t   \|u(\tau)\|_1^2 \dd \tau=\|u_0\|^2+2B_0t.
\end{align}

 To get more involved estimates, we first 
repeat a construction from \cite{K99} which evokes the maximum principle to bound the norm
 $|u(t,x)|$  of a solution $u(t,x)$ as in Theorem~\ref{T:1} in 
terms of a solution of a stochastic heat equation.

Let $\xi\in C^\ty(\R)$  be any function such that   
  $$
  \xi(r)=\begin{cases} 0\quad  \text{for $r\le \frac{1}{4}$, }     \\ r \quad \text{for $r\ge \frac{1}{2}$} .
\end{cases}$$ Writing $u_{\tau_M}^N$ in the polar form $u_{\tau_M}^N=re^{i\phi}$ and using the
  It\^o formula for $\xi(|u^M|)$ (see \cite{PZ92}, Section 4.5 and \cite{KS11}, Section 7.7), we get
\begin{align}\label{equation}
  \xi(r)&= \xi_0+\int_0^t   \Bigg[ \xi'(r)(\Delta r-r|\nabla \phi |^2) 
   +\frac{1}{2}\sum_{s\in\N^n} b_s^2\Big(\xi''(r)(  e^{i\phi}\cdot \varphi_s )^2    \nonumber
 \\&  +\xi'(r)\frac{1}{r}
(|\varphi_s|^2-( e^{i\phi}\cdot \varphi_s )^2)\Big) \Bigg]  \dd t
+  \Upsilon(t),
\end{align}
where $\xi_0=\xi(|u_0|)$,  $a\cdot b= \Re a\bar b$ for $a,b\in \C$ and $\Upsilon(t)$ is the 
real Wiener process 
\begin{align}\label{E:M}
 \Upsilon(t)=\sum_{s\in\N^n}\int_0^t\xi'(r) b_s \varphi_s (e^{i\phi}\cdot \dd \beta_s ).
 \end{align}
 Since
$\   |u|\le \xi +\tfrac12\,,$ 
   then to estimate $|u|$ it suffices to bound $\xi$. To do that 
 we compare  it with a real solution of the stochastic heat equation
\begin{align}
\dot v-\Delta v =\dot\Upsilon, \qquad
v(0)=v_0, \label{E:v1}
\end{align}where $v_0:=|\xi_0|$. We have that $v=v_1+v_2$, where $v_1$ is a solution of  (\ref{E:v1})
  with $\Upsilon:=0$, and $v_2$   is a  solution of (\ref{E:v1}) with $v_0:=0$. By the 
 maximum principle 
\begin{align}\label{E:v3}
\sup_{t\ge0}|v_1(t)|_\ty\le  |v_0|_\infty  \le  |u_0|_\infty.
\end{align}
To estimate $v_2$, we use the  following lemma established in 
\cite{K99} (see \cite{Kry, MR01,KNP03}  for more general results).  
\begin{lemma}\label{T:2*}
Let $v_2$ be a solution of  (\ref{E:v1}) with $\dot\Upsilon=\sum_s  b_s f^s(t,x)  \dot\beta_s(t)$ and 
$v_0=0$, where 
 progressively measurable functions $f^s(t,x)$  and real numbers $b_s$ are 
 such that $|f^s(t,x)|\le L$  for each $j$ and $t$ almost surely.
 Then a.s. $v_2$ belongs to $C(\R_+,\CK)$, and 
  for any $t\ge0$ and $p\ge1$ we have
\begin{align}\label{E:v4}
\E \sup_{s\in[t,t+T]}  |v_2(s)|^{2
p}_\ty\le  (C (T) L  B_*)^{2p} p^p.  \end{align} 
Moreover, 
$$\E\|v_2\mid_{([t,t+1]\times K)}\|^p_{C^{\theta/2,\theta}} \le C(p,\theta)
$$for any $0<\theta<1$, where $\|\cdot\|_{C^{\theta/2,\theta}} $ is the norm in the H\"older space of functions
on $[t,t+1]\times K$. 
\end{lemma} 
In \cite{K99}, this result is stated with a constant $C_p$ instead of $(CL  B_*)^{2p} p^p$ in the right-hand side of \eqref{E:v4}. Following the constants in the proof of \cite{K99}, one can see that $C_p=(CL  B_*)^{2p} \tilde C_p$, where $\tilde C_p$ is the constant in the Burkholder--Devis--Gundy inequality. By   \cite{B73}, we have $\tilde C_p\le C^p p^{p}$. 
\medskip

Using the definition of $\xi$ 
 we see that the noise   $\Upsilon$ defined by (\ref{E:M}) verifies 
the conditions of this lemma since the eigen-functions $\varphi_s$ satisfy
$|\varphi_s(x)| \le \left( {2}/{\pi}\right)^{\frac{n}{2}}$  for all $x\in K$.

Let us denote
$$
h(t,x)=\xi(r(t,x))-v(t,x).
$$
Since a.s. $u(t,x)$ is uniformly continuous on sets $[0,T]\times K$, $0<T<\infty$, then a.s. we can  find an open
 domain $Q=Q^\omega\subset [0,\infty)\times K$ with a piecewise smooth boundary $\p Q$ such that
$$
r\ge\frac{1}{2}\quad\text{in $Q$, }\quad r\le\frac{3}{4}\quad\text{ outside $Q$.}
$$
Then $h(t,x)$ 
is a solution of the following problem in $Q$
\begin{align}
  \dot h-\Delta h&=\frac{1}{2r}\sum_{s\in\N^n} b_s^2|\varphi_s|^2 -\left( r| \nabla \phi |^2 +\frac{1}{2r}\sum_{s\in\N^n}
 b_s^2( e^{i\phi}\cdot \varphi_s )^2\right)=:g(t,x), \label{E:h1}\\
  h|_{\p_+ Q}&=(r-v)|_{\p_+ Q}=:m,\label{E:h2}
      \end{align} where $\p_+ Q$ stands for the parabolic boundary, i.e., the part of the boundary of $Q$ where the
 external normal makes with  the time-axis an angle $\ge\pi/2$.
Note that $m(0,x)=0$. We write $h=h_1+h_2$, where $h_1$ is a solution of (\ref{E:h1}), 
(\ref{E:h2}) with $g=0$ and $h_2$ is a solution of (\ref{E:h1}), (\ref{E:h2}) with $m=0$. Since
each $ |\varphi_s(x)|$ is bounded by  $( {2}{\pi})^{{n}/{2}} $ and   $r\ge \frac{1}{2}$ in $Q$, then 
 $g(t,x)\le \left( {2}/{\pi}\right)^n B_0$ everywhere in $Q$.
 Now applying the maximum principle (see \cite{Lan}),  we obtain the inequality
\begin{align*}
  \sup_{t\ge0} |h_2(t)|_\ty\le CB_0,
      \end{align*}   
        (see Lemma~6 in \cite{K99}).  Therefore
  \begin{equation} \label{ots}
  |u(t)|_\infty\le \tfrac12 +|\xi(t)|_\infty\le \tfrac12 +CB_0+ |v_1(t)|_\infty+  |v_2(t)|_\infty + |h_1(t)|_\infty.
  \end{equation}
    To estimate $h_1$ we note that
      $$
      h_1(s,x)=\int_{\p_+ Q} m(\xi) G(s,x,\dd \xi),
      $$where $G(s,x,\dd \xi) $ is the Green function\footnote{It depends on $\omega$, as well as the set $Q$.
      All estimates below are uniform in $\omega$.}
       for the problem (\ref{E:h1}), (\ref{E:h2}) with $g=0$, which
      for any $(s,x)\in Q$ 
 is a probability measure in $Q$, supported by 
 $\p_+ Q$. Here we need the following estimate for $G$, proved 
 in \cite{K99}, Lemma~7, where
 $$
 Q_{[a,b]}:=Q\cap ([a,b]\times K).
 $$
 
      \begin{lemma}\label{L:reza}
      Let  $0\le s\le t$. Then for any $x\in K$  we have 
      $\ 
      G(t,x, Q_{[0,t-s]})= G(t,x, Q_{[0,t-s]} \cap {\p_+ Q} )
      \le 2^{\frac{n}{2}}e^{-\frac{n\pi^2}{4}s}.
      $
      \end{lemma}
      Since $r|_{\p_+ Q}\le \frac{3}{4}$,   we have  
      \begin{align}\label{0.1}
|h_1(t,x)|
 \le \frac{3}{4}+\int_{\p_+ Q}  |v_1(\xi)| G(t,x,\dd \xi) +
\int_{\p_+ Q}  |v_2(\xi)|G(t,x,\dd \xi).
      \end{align}Estimate   (\ref{E:v3}) implies 
      \begin{align}\label{0.1'}
              \int_{\p_+ Q}  |v_1(\xi)| G(t,x,\dd \xi)\le |u_0|_\ty.
      \end{align}
   Let us take a positive constant  $T$ and cover the segment $[0,t]$ by segments
   $I_1,\dots ,I_{j_T}$, where 
    $$
j_T=\Big[\frac{t}{T}\Big]+1,\qquad  I_j=[t-Tj,t-Tj+T].
$$
   To bound the last integral in \eqref{0.1}, we apply  Lemma~\ref{L:reza} as follows:
\begin{align*}       \int_{\p_+ Q} |v_2(\xi)| \, G(t,x,\dd \xi)  &\le 
\sum_{j=1}^{j_T} 
\int_{Q_{ I_j}} |v_2(\xi)|\, G(t,x,\dd \xi) \nonumber\\&\le 
2^{\frac{n}{2}}\sum_{j=1}^{j_T}  e^{-\frac{n\pi^2}{4}(j-1)T} \sup_{\tau\in  I_j}|v_2(\tau)|_\ty,
      \end{align*}
      where  $v_2(\tau)$ is extended by zero outside $ [0,t]$.  Denoting 
$$
\zeta_j=\sup_{\tau\in  I_j}|v_2(\tau)|_\ty,\qquad
Y=\sum_{j=1}^{j_T} e^{-2jT}\zeta_j,
$$ 
      and using  that  $n\pi^2/4>2$ we get
    \begin{equation}\label{0.111}
      \int_{\p_+ Q} |v_2(\xi)| \, G(t,x,\dd \xi)  \le CY.
      \end{equation}
      So by    \eqref{0.1'}   
 $\ 
 |h_1(t)|_\infty \le \frac{3}{4}+ |u_0|_\infty +CY  
 $.
 As  $|v_2(t,x)|\le \zeta_1\le CY$,   then   using \eqref{ots} and
   (\ref{E:v3})     we get   for any   $u_0\in H^m$  and any  $t\ge0$   that 
  the solution $u(t,x)$ a.s. satisfies 
 \begin{align}\label{E:h*x}
    |u(t,x)|
\le 2  |u_0|_\infty+ C B_0+2+C  Y. 
        \end{align} 
       
   Let us show that there are positive constants $c$ and $C$, not depending on $t$ and $u_0$, such that
   \begin{align}\label{E:h*x1}
   \E |u(t)|_\ty^2\le C e^{-ct}|u_0|_\ty^2+C \quad\text{for all $t\ge0$.}
   \end{align}
   Indeed,  since $v_1$ is a solution of the free heat equation, then
   \begin{align}\label{0.5}
    |v_1(t)|_\ty\le Ce^{-c_1t} |u_0|_\ty \quad \text{for $t\ge0.$}
\end{align} This relation, Lemma~\ref{L:reza} 
 and \eqref{E:v3} 
imply that
\begin{align}\label{0.4}
           \int_{\p_+ Q} | v_1(\xi)| G(t,x,\dd \xi) &\le  \int_{\p_+ Q_{[0,\frac{t}{2}]}} |v_1(\xi)| \, G(t,x,\dd \xi)  
            +   \int_{\p_+ Q_{[\frac{t}{2} ,t]}} |v_1(\xi)| \, G(t,x,\dd \xi) 
            \nonumber\\
         &  \le   \sup_{s\ge0}|v_1(s)|_\infty  G(t,x, Q_{[0,\frac{t}{2}]}) + \sup_{s\ge \frac{t}{2}}|v_1(s)|_\ty \nonumber\\
         &\le   |u_0|_\ty 2^{\frac{n}{2}}e^{-\frac{n\pi^2}{4}\frac{t}{2}}+ C e^{-{c_1 } t}|u_0|_\ty 
         \le C e^{-ct}|u_0|_\ty   .    
  \end{align} 
          By Lemmas~\ref{T:2*} and \ref{L:reza}
$$        \E \left|   \int_{\p_+ Q}  v_2(\xi) G(t,x,\dd \xi)\right|^2 \le     C   $$
for any $t\ge0$. Combining this with \eqref{ots}, \eqref{0.1},  \eqref{0.5} and \eqref{0.4}, we arrive at 
\eqref{E:h*x1}.

    Estimates \eqref{E:h*x} and \eqref{E:h*x1}
    are  used in the next section to get   bounds for  exponential moments of  $|u|_\infty$.

\subsection{Exponential moments of $|u(t)|_\infty$
}  \label{S:gn}

In this section, we strengthen bounds on polynomial moments of the random variables 
$\sup_{s\in [t,t+1]} |u(s)|_\ty^2$, obtained in Theorem~\ref{T:1}, 
to bounds on their exponential moments. As a consequence we
prove that  integrals $\int_0^T |u(s)|^2_\ty\,ds$ have linear growth as functions of $T$ and derive 
exponential estimates which characterise this growth. These estimates are crucially used in
Sections~\ref{S:3}-\ref{S:4} to prove that eq.~\eqref{E:1} defines a mixing Markov process.

\begin{theorem}\label{P:1} 
Under the assumptions of Theorem~\ref{T:1}, for any  $u_0\in H^m$, any $t\ge0$
 and $T\ge1$  the solution $u(t,x)$ satisfies 
the following estimates:
\begin{enumerate}
\item[(i)] There are constants $c_*(T)>0$ and $C(T)>0$, such that for any $ c\in(0,c_*(T)]$ we have 
  \begin{align}
\E \exp(c \sup_{s\in [t,t+T]}|u(s)|^2_\ty)&\le C(T)\exp{(5c|u_0|_\ty^2)}.
\label{E:6**} 
\end{align} 
  \item[(ii)] There are  positive constants $\la_0, C$  and $c_2$  such that
\begin{align}
   \E \exp(\la\int_0^t |u(s)|^2_\ty\dd s)\le C \exp{(c_1 |u_0|_\ty^2+c_2t)},
\label{E:6***} 
\end{align} 
for each $\lambda\le\lambda_0$, where $c_1={\text Const}\,\cdot\la$.
\end{enumerate}
\end{theorem}

{\it Proof.}
 {\it Step 1} (proof of 
(i)).  
Due to \eqref{E:h*x}, to prove \eqref{E:6**} we have to estimate exponential moments of $Y^2$.
 First let us show that for a suitable $C_2(T)>0$ we have
 \begin{align}
   \E \exp(c\sup_{s\in [t,t+T]} |v_2(s)|^2_\ty)&\le \frac{1}{1-cC_2(T)} 
   \quad\text{for any $t\ge0$ 
and $c<\frac{1}{C_2(T)}$.
}\label{E:6****} 
\end{align}
 Indeed, 
  using (\ref{E:v4}), we    get 
  \begin{align*} \E\exp(c\sup_{[t,t+T]} |v_2(s)|_\ty^2)&= \E \sum_{p=0}^\ty 
 \frac{c^p\sup_{[t,t+T]} |v_2(s)|_\ty^{2p}}{p!} \le \sum_{p=0}^\ty \frac{c^p  (C(T)  B_*)^{2p} p^p}{p!} 
\nonumber\\& 
 \le   \sum_{p=0}^\ty (ce(C (T)   B_*)^2)^p\le \frac{1}{1-ce (C(T)B_*)^2 }
\end{align*}  
since $p!\ge (p/e)^p$. 
  Thus we get (\ref{E:6****}) with $C_2:=e (C(T)  B_*)^{2}$. In particular,
  \begin{equation}\label{yy}
  \E e^{{c'}\zeta_j^2}\le  \big(1-{c'}C_2(T)\big)^{-1}\qquad \forall\, c'\le c.
  \end{equation}

Next we note that since 
 $$
Y^2\le C^2 \left(\sum_{j=1}^{j_T}e^{- j}  \Big( e^{-j} \zeta_j \Big)\right)^2\le 2C^2\sum_{j=1}^{j_T}e^{-2j}\zeta_j^2
 $$
 by Cauchy-Schwartz  (we use that  $T\ge1$), then  
 $$
 \E e^{{c'} Y^2}\le \E \prod_{j=0}^{j_T} e^{2{c'}C^25^{-j}\zeta_j^2},
 $$
 as $e^2>5$. Denote $p_j=\alpha 2^j, j\ge0$. Choosing $\alpha\in (1, 2)$ in a such a way that
 $\sum_{j=0}^{j_T} ({1} /{p_j})=1$, using the  H\"older inequality with these $p_j$'s and (\ref{yy}), we find that
 \begin{equation}\label{Y}
 \begin{split}
 \E e^{{c'}Y^2}&
\le  \prod_{j=0}^{j_T} \left( \E e^{2p_j {c'}C^25^{-j}\zeta_j^2}\right)^{\frac{1}{p_j}}
\le  \prod_{j=0}^{j_T}   \left( \E e^{2  {c'}C^2 \zeta_j^2}\right)^{\frac{1}{p_j}}\\ &
\le   \prod_{j=0}^{j_T}  \big(1-  {c'} C_3(T)\big)^{-\frac{1}{p_j}} 
= \exp\left(-\sum_{j=0}^{j_T}  p_j^{-1}\, {\ln (1-  {c'}  C_3)}  \right)   \le e^{{c'}C_4(T)},
 \end{split}
\end{equation} 
if $2c' C^2\le c$ and  $ {c'}\le \big(2C_3(T)\big)^{-1}$. In view of \eqref{E:h*x},
this implies (\ref{E:6**}).
\medskip

 {\it  Step 2}. Now we show that  for any $A\ge1$  there  is a time $T(A)$  such that for  $T\ge T(A)$  we have
  \begin{align}
\E \exp\Big(c( \sup_{s\in [0,T]}|u(s)|^2_\ty+ A|u(T)|_\ty^2)\Big)&\le \tilde C\exp{(6c|u_0|_\ty^2)} \label{E:61**} 
\end{align} 
for any $c\in (0, \tilde c]$, where    $\tilde C $ and $\tilde c$ depend on $A$ and $T$.
 Indeed, due  to \eqref{ots}   and \eqref{0.5}, 
$$
 |u(T)|_\infty \le 
 2+Ce^{-cT}|u_0|_\infty 
+CB_0+ |v_2(T)|_\infty +|h_1(T)|_\infty.
$$
By \eqref{0.1}, \eqref{0.4} and \eqref{0.111},
$
\ |h_1(T)|_\infty\le \tfrac34 +CY+Ce^{-c'T}|u_0|_\infty.
$
Therefore  choosing a suitable $T = T(A)$ we achieve that 
$$
cA|u(T)|^2_\infty \le c(C_1 A+C_2AY^2+|u_0|^2_\infty)+2c A|v_2(T)|_\infty^2.
$$
Using H\"older's inequality we see that the cube of the l.h.s. of \eqref{E:61**} is
bounded by
$$
C(A) e^{3c|u_0|^2_\infty}\E e^{3cC_2AY^2}\,\E e^{6cA|v_2(T)|^2_\infty}\,
 \E e^{3c\sup_{s\in[0,T]}|u(s)|^2_\infty}.
$$
Taking $c\le c(A)$ and using \eqref{Y}, \eqref{E:6****} and \eqref{E:6**} we estimate
the product  by 
$
C(A,T)e^{3c|u_0|^2_\infty}\, e^{15c|u_0|^2_\infty}. 
$
 This implies  (\ref{E:61**}).
\medskip

 {\it Step 3} (proof of
(ii)). 
  Let $T_0\ge1$ be  such that  (\ref{E:61**}) holds with $A=6$.      Let $c>0$ and $C>0$ be the constants
 in (\ref{E:6**}), corresponding to $T=T_0$, and let $\la\le  {c}/{ T_0}$.  It suffices to prove  \eqref{E:6***} for $t=T_0 k$, $k\in\N$, 
 since this result implies  \eqref{E:6***}  with any $t\ge0$ if we modify the constant $C$.
 By the   Markov property,
\begin{align*}
X_\lambda:=
   \E_{u_0} \exp\Big(\la\int_0^{{T_0k}} |u(s)|^2_\ty\dd s\Big)=   \E_{u_0}&\Big( \exp(\la\int_0^{{T_0(k-1)}} 
|u(s)|^2_\ty\dd s)\nonumber\\
&\times\E_{u({T_0({k-1})} )} \exp(\la\int_0^{T_0} |u(s)|^2_\ty\dd s)\Big),
\end{align*}
and by (\ref{E:6**}) 
\begin{align*}
\E_{u({T_0({k-1})})} \exp\Big(\la\int_0^{T_0} |u(s)|^2_\ty\dd s\Big)
\le C \exp\big(5 \la T_0 |u({T_0({k-1})})|_\ty^2\big).
\end{align*} 
Combining  these  two relations we get 
$$
X_\lambda\le
C \E_{u_0}  \exp\big ( \la \int_0^{{T_0(k-1)}} |u(s)|^2_\ty\dd s+6 T_0 |u({T_0({k-1})}|_\ty^2  \big).
$$
Applying again  the   Markov property and using (\ref{E:61**}) with $A=6$ and $c=\lambda T_0$ 
 we obtain
\begin{align*}
X_\lambda &\le   C\E_{u_0}\Big( \exp(\la\int_0^{{T_0(k-2)}} 
|u(s)|^2_\ty\dd s)\nonumber\\
&\quad\times\E_{u(({T_0({k-2})})}     \exp \big( \la T_0(\sup_{0\le s\le T_0}|u(s)|^2_\infty   +6  |u({T_0})|_\ty^2)\big)\Big)  \nonumber\\
&\le C^2\E_{u_0}  \exp\Big (\la\int_0^{{T_0(k-2)}} |u(s)|^2_\ty\dd s  +  
 6  \la T_0  |{u({T_0({k-2})}})|_\ty^2 \Big). 
 \end{align*} 
 Iteration gives
   \begin{align*}   
   X_\lambda \le   C^m\E_{u_0}   \exp 
 \Big(\la\int_0^{{T_0(k-m)}} |u(s)|^2_\ty\dd s \nonumber + 6 \la T_0  |{u({T_0({k-m})}})|_\ty^2 \Big),
  \end{align*} 
  for any $m\le k$.   When  $m=k$, this relation proves (\ref{E:6***}) with $t=kT_0, C=1, c_1=6\la T_0$ and a 
  suitable $c_2$.
\qed
\smallskip

In the lemma below by $c_1, c_2$ and $\lambda_0$ we denote the constants from 
Theorem~\ref{P:1}(ii). 

 \begin{lemma} \label{P:3}
  For any  $u_0\in  H^m $ the solution $u(t,x)$ satisfies the following estimate for any $\rho\ge0$   
      \begin{align}\label{E:3a}
\pP\{ \sup_{t\ge 0}\left(\int_0^{t} |u(s) |_\ty^2\dd s-Kt \right)\ge\rho\}\le C' \exp(  c_1|u|^2_\ty-  \la \rho),
\end{align}  where $C'$ is an absolute constant, 
  $K=\lambda^{-1}(c_2+1)$ and $\lambda$ is a suitable constant from $(0,\lambda_0]$.
      \end{lemma}
\begin{proof} 
For any real number $t$ denote $\lceil t\rceil =\min\{n\in\Z: n\ge t\}$.  Then
$$
\Big\{ \big(\int_0^t|u|^2_\infty\,ds-Kt\big) \ge\rho\Big\}\subset
\Big\{\big(\int_0^{\lceil t\rceil } |u|^2_\infty\,ds-K\lceil t\rceil   \big) \ge\rho -K\Big\}.
$$
So it suffices to prove \eqref{E:3a} for integer $t$ since then the required inequality follows with a modified constant $C'$. 
 Accordingly below we replace $\sup_{t\ge0}$ by $\sup_{n\in\N}$.
By the Chebyshev inequality and estimate   \eqref{E:6***}   we have
\begin{align*}
\pP\Big\{ \sup_{n\in\N}\left(\int_0^{n} |u(s) |_\ty^2\dd s-Kn \right)\ge\rho\Big\}
&\le\sum_{n\in \N}\pP\Big\{  
 \int_0^{n} |u(s) |_\ty^2\dd s  \ge\rho+Kn \Big\} \nonumber\\
  &\le \sum_{n\in \N}
 \exp(-\la (\rho+Kn)) C \exp{(c_1 |u_0|_\ty^2+c_2n)}\nonumber\\
 &\le   C\exp( -\lambda\rho +c_1|u_0|^2_\infty)
  \sum_{n\in \N}\exp{(-n)}\\
  &= C'\exp( c_1|u_0|^2_\infty-\lambda\rho)
\end{align*}
since $\lambda K-c_2 =1$.
 This proves (\ref{E:3a}).
 \end{proof}

\section{Markov Process in $\CK$.}\label{S:3}

The   goal of this section is to construct   a family of Markov processes, associated with eq.~(\ref{E:1}) 
  in the space $\CK$. To this end we first prove  a well-posedness result in that space.

\subsection{Existence and uniqueness of solutions}
Let $u_0\in \CK$. Denote by $\Pi_m: {H}\ri\C^m$ the usual Galerkin
 projections and set $\eta^m:=\Pi_m\eta=:\frac{\p}{\p t} \zeta^m$. Let $u_0^m\in C^\ty$ be such that 
$|u_0^m-u_0|_\ty\ri 0$ as $m\ri \ty$ and $|u_0^m|_\ty\le |u_0|_\ty+1$. Let $u^m$  be a solution of  (\ref{E:1}),
 (\ref{E:5}) with $\eta=\eta^m$ and $u_0=u_0^m$, existing by Theorem~\ref{T:1}. 
 \smallskip
 
 Fix any $T>0$. 
 For  $p>1$, $ \alpha\in (0,1)$ and a Banach space $X$, let  
   $W^{\alpha,p}([0,T],X)$ be the space of all $u\in L^{p}([0,T],X)$ such that
$$
\|u\|_{W^{\alpha,p}([0,T],X)}^p:=\|u\|_{W^{p}([0,T],X)}^p+\int_0^T\!\!\!\int_0^T
\frac{\|u(t)-u(\tau)\|_X^p}{|t-\tau|^{1+\alpha p}}\dd \tau \dd t<\ty.
$$
 Let us define  the spaces
 \begin{align*} 
\UU&:= L^2([0,T],H^1)\cap W^{\alpha,4}([0,T],H^{-1}),\\
\VV&:= L^2([0,T],H^{1-\e})\cap C([0,T],H^{-2}),
\end{align*}where $\alpha\in (\frac{1}{4},\frac{1}{2})$ and $\e\in(0,\tfrac12)$ are fixed. Then
\begin{equation}\label{comp}
\text{
 space $\UU$ is compactly
 embedded into~$\VV$. }
 \end{equation}
 Indeed, by Theorem 5.2 in \cite{Lio69}, $\UU\Subset L^2([0,T],H^{1-\e})$. On the other hand,  $ W^{\alpha,4}([0,T],H^{-1})\subset C^{\alpha-\frac{1}{4}}([0,T],H^{-1})$ 
and $H^{-1}\Subset H^{-2}$.
\begin{lemma}\label{L:e}
For $m\ge1$ let $M_m$ be the law of the solution
  $\{u^m\}$, constructed above. Then 
\begin{enumerate}
\item[(i)]  The sequence $\{M_m\}$ is tight in $\VV$.
\item[(ii)] Any limiting measure $M$ of  $M_m$  is the  law of a 
weak solution $\tilde u(t), \ 0\le t\le T$,   of (\ref{E:1}), (\ref{E:5}).
 This solution  satisfies  (\ref{E:6}) for $0\le t\le T-1$ and 
  (\ref{E:31}), (\ref{E:6**}), 
  (\ref{E:3a}) for $ 0\le t\le T$. 
 \item[(iii)] If $1\le t\le T-1$, then for any $0<\theta<1$ and any $q\ge1$ we have 
\begin{equation}\label{3.33}
\E\|\tilde u\mid_{ ([t,t+1]\times K)}\|^q_{C^{\theta/2,\theta}} \le C(q,\theta,|u_0|_\infty).
\end{equation}
\end{enumerate}
\end{lemma}

\begin{proof}
The process $u^m$ satisfies the following equation with probability~$1$ 
 \begin{align*}  u^m(t) = u^m_0 +\int_0^t (\Delta u^m-i |u^m|^2u^m)\dd s+ \zeta^m=:V^m+\zeta^m.
  \end{align*} 
  Using    (\ref{E:6}) and (\ref{E:31}), we get
  \begin{align}\label{E:s2*} 
   \E \|V^m\|_{W^{1,2}([0,T],H^{-1})}^2\le   C. 
 \end{align} It is well known that 
for any $p>1$ and $\alpha\in(0,\frac{1}{2})$, we have 
  \begin{align}\label{E:s3*} 
   \E \|\zeta^m\|_{W^{\alpha,p}([0,T],H)}^2&\le C
 \end{align} (e.g., see \cite{KS11}, Section 5.2.1).
 Combining (\ref{E:s2*}) and (\ref{E:s3*}), we get 
 $$
  \E \|u^m\|_{W^{\alpha,4}([0,T],H^{-1})}^2\le  C\E \|V^m\|_{W^{1,2}([0,T],H^{-1})}^2+
 C\E \|\zeta^m\|_{W^{\alpha,4}([0,T],H)}^2\le C.
 $$
 Jointly with (\ref{E:31})  this estimate    implies  that $\E\|u^m\|^2_{\UU}\le C_1$ for each $m$
 with a suitable $C_1$. Now (i) holds by \eqref{comp} and  the   Prokhorov theorem.
 
  Let us prove (ii). 
Suppose that $M_m$ converges weakly to $M$ in $\VV$.
 By Skorohod's embedding theorem, there is a probability space  $(\tilde\Omega,\tilde\FF, \tilde\pP)$, 
and $\VV$-valued random variables $\tilde u^m$ and $\tilde u$ defined on it such that each $\tilde u^m$ is 
distributed as $M_m$, $\tilde u$ is distributed as $M$ and 
 $\pP$-a.s. we
 have $\tilde u^m\ri \tilde u$ in $\VV$. 

Since $\VV\subset L_2([0,T]\times K)=: L_2$, then $\tilde u^m\to \tilde u$ in $L_2$, a.s. For any 
$R\in(0,\infty]$ and $p, q\in [1,\ty)$  consider the functional $f^p_R$, 
$$
f^p_R(u)=\big| |u|^q\wedge R\big|_{L^p([t,t+1]\times K)}\le|u|^q_\infty.
$$
Since for $p,R<\infty$ it is continuous in $L_2$, then by \eqref{E:6} we have $\E (f^p_R(\tilde u))\le C_q$
for $p,R <\infty$. 
As for each $v(t,x)\in L^\ty([t,t+1]\times K)$ the function $[1,\infty]\ni p\mapsto |v|_{L^p([t,t+1]\times K)}\in [0,\infty]$ is 
continuous and non-decreasing, 
then sending $p$ and $R$ to $\infty$ and using the monotone convergence theorem, we get
$\ \E\sup_{s\in [t,t+1]} |\tilde u(s)|^q_\ty\le C_q.$ I.e., $\tilde u$ satisfies \eqref{E:6}.

By \eqref{E:31} for each $m$ and $N$  we have 
$$
 \E \|\Pi_N\tilde u^m(t)\|^2+2\E \int_0^t   \|\Pi_N \tilde u^m (\tau)\|_1^2 \dd \tau  \le\|u^m_0\|^2+B_0t.
$$Passing to the limit as $m\ri\ty$ and then $N\ri \ty$ and using the monotone convergence theorem, we obtain that
$\tilde u$ satisfies \eqref{E:31}, where the equality sign is replace by $\le\,$.  We will call this estimate
\eqref{E:31}${}_{\le}\,$.

 By the same reason  (cf. Lemma~1.2.17 in \cite{KS11}) the process 
 $\tilde u(t)$  satisfies \eqref{E:6**} and (\ref{E:3a}).


Since $\tilde u^m$ is a weak solution
of the equation,  then
 \begin{align}\label{u^m}  
 \tilde u^m(t) -  u^m_0 -\int_0^t (\Delta \tilde u^m-i |\tilde u^m|^2\tilde u^m)\dd s= \tilde \zeta^m ,
  \end{align} 
 where $\tilde \zeta^m$ is distributed as the process $\zeta$. Using the 
 Cauchy--Schwarz inequality and \eqref{E:6}, we get 
  \begin{align*} 
\E  \int_0^T & \big \| |\tilde u^m|^2\tilde u^m  -|\tilde u |^2\tilde u\big \|\dd s
\le C\, \E  \int_0^T   \| (\tilde u^m-\tilde u)(|\tilde u^m|^2+|\tilde u|^2)  \|\dd s
\\&\le C\, \E \sup_{t\in[0,T]}(|\tilde u^m(t)|_\ty^2+|\tilde u(t)|_\ty^2)  \int_0^T   \| \tilde u^m-\tilde u  \|\dd s
\\&\le C\sqrt{T}\left(  \E \sup_{t\in[0,T]}(|\tilde u^m(t)|_\ty^4+|\tilde u(t)|_\ty^4)\right)^\frac{1}{2}
          \left ( \E \int_0^T   \| \tilde u^m-\tilde u  \|^2\dd s\right)^\frac{1}{2}
\\&\le C(T,|u_0|_\ty)\left (\E \int_0^T   \| \tilde u^m-\tilde u  \|^2\dd s\right)^\frac{1}{2} .
  \end{align*} 
  Since the r.h.s. goes to zero when $m\to\infty$, then  for a suitable subsequence $m_k\ri\ty$ we have a.s.
  $$\Big\|\int_0^t   |\tilde u^{m_k}|^2\tilde u^{m_k}\dd s - \int_0^t   |\tilde u|^2\tilde u\dd s\Big\|_{C([0,T], L^2)}\ri 0\quad \text{as $k\ri \ty.$}
  $$
  Therefore the l.h.s. of  \eqref{u^m} converges to 
  $\big( \tilde u (t) -   u _0 -\int_0^t (\Delta \tilde u -i |\tilde u |^2\tilde u )\dd s\big)$ 
  in the space $C([0,T], H^{-2})$ over the sequence $\{m_k\}$, a.s. So a.s. 
   there exists a limit $\lim \tilde \zeta^{m_k}(\cdot)=\tilde \zeta(\cdot)$, and 
    \begin{align}  \label{300}
 \tilde u (t) -   u _0 -\int_0^t (\Delta \tilde u -i |\tilde u |^2\tilde u )\dd s= \tilde \zeta(t).
  \end{align}
  We immediately get that  $\tilde \zeta(t)$ is a Wiener process 
  in $H^{-2}$, distributed as the process $\zeta$.  Let $\tilde\FF_t, \  t\ge0$, be a sigma-algebra, generated 
  by $\{\tilde u(s), 0\le s\le  t\}$ 
  and the zero-sets of the measure $\tilde\pP$. From \eqref{300}, $\tilde\zeta(t)$ is  $\tilde\FF_t$-measurable. So $\tilde\zeta(t)$ is a 
  Wiener process on the filtered probability space $(\tilde\Omega,\tilde\FF, \{\tilde \FF_t\},   \tilde\pP)$, distributed as $\zeta$. 
    \smallskip

 Since $\tilde u(t,x)$ satisfies \eqref{300}, we can write
 $\ 
 \tilde u=u_1+u_2+u_3,
 $
 where $u_1$ satisfies \eqref{E:v1} with $\dot\Upsilon=0, v_0=u_0$; $u_2$ satisfies \eqref{E:v1} with
 $\dot\Upsilon=-i|\tilde u|^2\tilde u, v_0=0$  and $u_3$ satisfies \eqref{E:v1} with $\Upsilon=\tilde\zeta, v_0=0$.
 Now Lemma~\ref{T:2*} and the parabolic regularity imply  that $\tilde u\in C([0,T];C_0(K))$, a.s. As  $\tilde u$ satisfies 
 \eqref{E:31}${}_{\le}\,$, 
 then $\tilde u\in\HH([0,T])$ a.s. Since clearly $\tilde u(0)=u_0$ a.s., then 
  $\tilde u$ is a weak solution of \eqref{E:1},  \eqref{E:5}. 
  
  Regarding $\tilde u(t)$ as an Ito process in the space $H$, using \eqref{E:6} 
   and applying to $\| \tilde u(t)\|^2$
   the Ito formula in the form, given in \cite{KS11}, we see that   $\| \tilde u(t)\|^2$ satisfies the relation, 
   given by the displayed formula above  \eqref{E:31}. Taking the expectation we recover for $\tilde u$ the
   equality \eqref{E:31}.
 
  \smallskip

   It remains to prove (iii).   Functions $u_1$ and $u_3$  meet 
  \eqref{3.33} by  Lemma~\ref{T:2*} and the parabolic regularity. Consider $u_2$.  Since $u_2=\tilde u-u_1-u_3$, 
 then $u_2$ satisfies \eqref{E:6}. Consider restriction of $u_2$ to the cylinder $[t-1,t+1]\times K$. Since $u_2$ satisfies 
 the heat equation, where the r.h.s. and the Cauchy data at $(t-1)\times K$   are bounded functions, then by 
 the parabolic regularity restriction of   $u_2$  to $[t,t+1]\times K$ also meets \eqref{3.33}. 
 \end{proof}

   The pathwise  uniqueness property holds for the constructed  solutions:

\begin{lemma}\label{L:1}
Let $u(t)$ and  $v(t)$, $t\in[0,T]$,  be processes  in the space $C_0(K)$,
defined on the same  probability space, and let $\zeta(t)$ be a Wiener process, defined on the 
 same space and distributed as $\zeta$ in \eqref{E:2*}.  Assume that a.s.   trajectories  of  $u$ and $v$ 
  belong to  ${\cal H}([0,T] )$and  satisfy   (\ref{E:1}), (\ref{E:5}). 
 Then  $u(t)\equiv v(t)$ a.s.
\end{lemma}
\begin{proof} 
 For any $R>0$ let us introduce the stopping time 
\begin{align}\label{E:st}
\tau_R=\inf\{t\in[0,T] :|u(t)|_\ty\vee|v(t)|_\ty\ge R\},
\end{align} and consider the  stopped solutions
 $u_R(t):=u(t\land\tau_R)$ and $v_R(t):=v(t\land\tau_R)$. Then $w:=u_R-v_R$ satisfies  
$$
\dot w-\Delta w+i(|u_R|^2u_R-|v_R|^2v_R)= 0, \quad w(0)=0.
$$Taking the scalar product in $H$ of this equation with $w$ and applying the  
  Gronwall inequality, we get that $w(t)\equiv0$.  Since $u,v\in\HH([0,T])$, then $\tau_R\to T$, a.s.
  Therefore    $u_R\ri u$ and 
$v_R\ri v$ 
a.s. as $R\ri\ty$. This completes the proof.
\end{proof}
By  the Yamada--Watanabe arguments (e.g., see   \cite{KS91}), existence of a weak solution plus
 pathwise uniqueness implies the existence of a unique strong solution $u(t), 0\le t\le T$. Since $T$
 is any positive number, we get

\begin{theorem}\label{T:2}
Let $u_0\in \CK$. Then   problem (\ref{E:1}), (\ref{E:5}) has 
 a unique strong solution  $u(t), \ t\ge0$. This solutions satisfies relations 
(\ref{E:6}),  (\ref{E:31}), (\ref{E:6**}) and    (\ref{E:3a}); for $t\ge1$ it also  satisfies    (\ref{3.33}).
\end{theorem}

\subsection{Markov process}\label{s:32}
 Let us denote by $u(t)=u(t,u_0)$   the unique solution  solution of  (\ref{E:1}), corresponding to an 
initial condition $u_0\in \CK$.  Equation (\ref{E:1}) defines a family of  Markov process in the space $\CK$    parametrized by
  $u_0$.  For any $u\in \CK$
and $\Gamma\in \BBBBB (\CK)$, we set
$P_t(u,\Gamma)=\pP\{u(t,u)\in\Gamma\}$. The Markov operators
corresponding to the process $u(t)$ have the form
$$
\PPPP_t f(u)=\int_{\CK}P_t(u,\dd
v)f(v),\,\,\,\,\,\,\,\PPPP^*_t\mu(\Gamma)=\int_{\CK}
P_t(u,\Gamma)\mu(\dd u),
$$ where $f\in C_b(\CK)$ and $\mu\in\PP(\CK).$

 \begin{lemma}\label{L:2}
The Markov process associated with  (\ref{E:1}) is Feller.
\end{lemma}
 \begin{proof}
 We need to prove that $\PPPP_t f\in C_b(\CK)$ for any $f\in C_b(\CK)$ and $t>0$.
To this end, let us take any $u_0, v_0\in \CK$, and let $u$ and $v$ be the corresponding solutions of (\ref{E:1})
 given by Theorem~\ref{T:2}. Let us take any 
$\ 
R>R_0:=|u_0|_\ty\vee|v_0|_\ty.
$
 Let $\tau_R$ be the  stopping time defined by (\ref{E:st}), and let $u_R(t):=u(t\land\tau_R)$ and
 $v_R(t):=v(t\land\tau_R)$ be the stopped solutions.
 Then
\begin{align*}
|\PPPP_t f(u_0)- \PPPP_t f(v_0)|&\le \E |f(u)- f(u_R)|+ \E |f(v)-  f(v_R)|\nonumber\\&\quad+ \E| f(u_R)- 
 f(v_R)|=:I_1+I_2+I_3.
\end{align*}
By (\ref{E:6}) and the Chebyshev inequality, we have
\begin{align*}
\max\{I_1,I_2\}&\le 2 |f|_{\ty}\pP\{t>\tau_R\}\le 2 |f|_{\ty}\pP\{ U(t) \vee V(t) >R\}\nonumber\\&\le 
\tfrac{4}{R} \, |f|_{\ty}\sup_{|u_0|_\ty\le R_0} \E \,U(t)\ri0 \quad\text{ as $\quad R\ri\ty,$}
\end{align*}
where $U(t)=\sup_{s\in[0,t]}|u(s)|_\ty$ and $V(t)$ is defined similarly.
To estimate $I_3$, notice that $w=u_R-v_R$ is a solution of
 \begin{align*}
\dot w-\Delta w+i (|u_R|^2u_R -|v_R|^2v_R)=0, \quad
w(0)=u_0-v_0=:w_0.\end{align*}   We rewrite this in the Duhamel form 
$$ w=e^{t\Delta}w_0-i\int_0^t e^{(t-s)\Delta } (|u_R|^2u_R -|v_R|^2v_R)\dd s.
 $$
 Since, by the maximum principle, 
$
 |e^{t\Delta} z|_\ty\le |z|_\ty ,
  $ then
 $$
 |w|_\ty\le    |w_0|_\ty+ \int_0^t ||u|^2u -|v|^2v|_\ty\dd s\le   |w_0|_\ty+3\int_0^t (|u|_\ty^2 +|v|_\ty^2)|w|_\ty\dd s.
 $$By the Gronwall inequality, 
 $\ 
 I_3\le \E |w|_\ty\le  |w_0|_\ty e^{tC_R}\ri0$ as $|w_0|_\ty\ri0$.
 Therefore the function $\PPPP_t f(u)$ is continuous in $u\in C_0(K)$, as stated. 
  \end{proof}

A measure $\mu\in\PP(\CK)$ is said to be stationary for  
eq.~(\ref{E:1}) if $\PPPP_t^*\mu=\mu$ for
every $t\ge0$. The following theorem is proved in the standard way by applying the  Bogolyubov--Krylov argument
 (e.g. see in \cite{KS11}).
   \begin{theorem}\label{T:BK}
Equation (\ref{E:1}) has at least one stationary measure $\mu$, satisfying  
 $\ 
 \int_{H^1} \|u\|^2_1\mu(\dd u)=\frac12 B_0$ and 
 $  \int_{\CK}  e^{c|u |^2_\ty}\,
 \mu(\dd u)<\infty
$
for any $c<c_*$, where $c_*>0$ is the constant in assertion (i) of Theorem~\ref{P:1}.
\end{theorem}

\subsection{Estimates for some hitting times}\label{S:h}

 For any $d,L,R>0$ we 
    introduce the following     hitting times   for a solution $u(t)$ of \eqref{E:1}:
     \begin{align*}
 \tau_{1,d,L}&:=\inf\{t\ge0: \|u(t)\|\le d, |u(t)|_\ty\le L\},\\
 \tau_{2,R}&:=\inf\{t\ge0: |u(t)|_\ty\le R\}.
 \end{align*}    
 \begin{lemma}\label{L:gn}
There is a constant $L>0$ such that for any $d>0$ we have
 \begin{align}\label{E:tau2}
\E e^{\gamma\tau_{1,d,L}}\le C(1+|u(0)|^2_\ty), 
  \end{align}
 where $\gamma$ and $C$   are suitable positive constants,
depending on $d$ and $L$.
\end{lemma}
 
 It is well known that inequality (\ref{E:tau2})   follows the from  two statements below  (see 
  Proposition 2.3 in \cite{Shi04} or Section 3.3.2 in \cite{KS11}).
  \begin{lemma}\label{L:(a)}
  There are positive constants $\delta, R$ and $C$ such that
\begin{align}\label{E:taur2} 
\E e^{\delta\tau_{2,R}}\le C(1+|u(0)|^2_\ty).
\end{align}
  \end{lemma}
\begin{lemma}\label{L:(b)}
For any $R>0$ and $d>0$ there is a non-random 
 time $T>0$ and positive  constants $p $  and $L$  such that
\begin{align*}
 \pP\{ u({T}, {u_0}) \in B_{{H}}(d)\cap  B_{\CK}(L)\}\ge p\,\,\,\,\text{for any ${u_0}\in B_{\CK}(R)$}.
\end{align*}
  \end{lemma}
   \begin{proof}[Proof of Lemma~\ref{L:(a)}]
   Let us  consider the function 
$
F(u)=\max(|u|_\ty^2,1).
$
We claim that  this is a Lyapunov function for eq.~\eqref{E:1}. That is, 
\begin{equation}\label{E:L11}
\E F(u({T}, u)  )\le aF(u)\,\,\,\,\,\,\text{for $|u|_\ty\ge R'$,}
\end{equation}
for  suitable $a\in(0,1)$, $T>0$ and $R'>0$. 
Indeed, let  $|u|_\ty\ge R'$ and $T>1$.  Since $F(u)\le 1+|u|_\infty^2$, then 
$$
\E F(u({T},u)) 
\le 1+\E |u({T},u)|_\ty^2\le1+   Ce^{-cT}|u|_\ty^2+C,
$$
where we used (\ref{E:h*x1}). 
 This implies (\ref{E:L11}). 
Since due to  (\ref{E:h*x1})  for $|u|_\ty<R'$ and any $T>1$
we have 
$
\E F(u({T}, u)   )\le C'$
then \eqref{E:taur2}  follows by a standard argument with Lyapunov function (e.g., see Section~3.1 in 
\cite{SHI08}). 
\end{proof}

 \begin{proof}[Proof of Lemma~\ref{L:(b)}] 
 {\it Step 1.} Let us write $u(t)=v(t)+z(t)$, where $z$ is a solution of \eqref{E:v1} with $v_0=0$, i.e.,
 $$
 z=
 \sum_{s\in\N^n}\int_0^te^{(t-\tau)\Delta } b_s 
\varphi_s\dd \beta_s^\om.
 $$
Then
 \begin{align}
\dot v-\Delta v +i|v+z|^2(v+z)=0,\qquad v(0)=u_0. \label{E:11}  
\end{align} 
 Clearly for any $\delta\in(0,1]$ and $ T>0$ we have
 \begin{equation*}
 \pP \Omega_\delta>0\,,\qquad \text{where}\quad 
 \Omega_\delta=\{\sup_{0\le t\le T}|z(t)|_\infty<\delta\}.
 \end{equation*}
 
 \noindent
 {\it Step 2.}
 Due to \eqref{E:11},
 \begin{equation}\label{.0}
\dot v-\Delta v+i|v|^2v=L_3,\qquad
(t,x)\in Q_T=[0,T]\times K,
\end{equation}
 where $L_3$ is a cubic polynomial in $v,\bar v, z,\bar z$ such that every its monomial contains 
 $z$ or $\bar z$. Consider the function $r=|v(t,x)|$. Due to \eqref{.0}, for $\omega\in\Omega_\delta$
and  outside the zero-set $X=\{r=0\}\subset Q_T$ the function $r$ satisfies the parabolic
inequality
\begin{equation}\label{.6}
\dot r -\Delta r\le C\delta (r^2+1),\qquad
r(0,x)=|v(0,x)|\le R+1. 
\end{equation}
Define $\tau=\inf\{t\in[0,T]:|r(t)|_\infty \ge  R+2\}$, where $\tau=T$ if the set is empty.  Then $\tau>0$ 
and for $0\le t\le\tau$  the r.h.s. in \eqref{.6} is 
$\ \le C\delta((R+2)^2+1)=\delta C_1(R)$. Now consider the function 
$$
\tilde  r(t,x)=r-(R+1)-t\delta C_1(R).
$$
Then $\tilde r\le 0$ for $t=0$ and for $(t,x)\in\partial(Q_T\setminus K)$. Due to \eqref{.6} and the
definition of $\tau$,  for $(t,x)\in Q_\tau\setminus X$ this function satisfies 
$$
\dot {\tilde r}-\Delta \tilde r\le
C\delta(r^2+1)-\delta C_1(R)\le 0.
$$
Applying the maximum principle \cite{Lan}, 
 we see that $\tilde r\le 0$ in $Q_\tau\setminus K$. So for 
$t\le\tau$ we have $r(t,x)\le (R+1)+t\delta C_1(R)$.  Choose $\delta$  so small that 
$T\delta C_1(R)<1$. Then $r(t,x)<R+2$ for $t\le\tau$. So $\tau=T$ and we have proved that 
  \begin{equation}\label{.1}
 |v(t)|_\infty =  |r(t)|_\infty\le R+2
 \quad  \forall \, 0\le t\le T \;\; \text {if}\;\;
 \delta\le\delta(T,R), \; \omega\in\Omega_\delta. 
 \end{equation}

  \noindent
  {\it Step 3.} It remains to estimate $\|v(t)\|$. To do this we first define $v_1(t,x)$ as a solution of eq.~\eqref{E:1} with
  $\eta=0$ and $v_1(0)=u_0$. Then 
  \begin{equation}\label{.3}
  \|v_1(t)\|\le e^{-\alpha_1t}\|u_0\|,\qquad |v_1(t)|_\infty\le |u_0|_\infty\le R,
 \end{equation}
 since outside its zero-set the function $|v_1(t,x)|$ satisfies a parabolic inequality with the maximum principle (namely,
  eq.~\eqref{.6}  with $\delta=0$). 
  \smallskip
 
  \noindent
  {\it Step 4.}  Now we estimate $w=v-v_1$.  This function solves  the following equation:
  $$
\dot w-\Delta w + i \big(|v+z|^2(v+z)-|v_1|^2v_1\big)=0,\qquad w(0)=0.
$$ 
Denoting $X=w+z$ (so that $v+z=X+v_1$), 
we see that the term in the brackets is a cubic polynomial $P_3$ of the variables 
$X, \bar X, v_1$ and $\bar v_1$, such that every its monomial contains $X$ or $\bar X$. Taking the $H$-scalar product 
of the $w$-equation with $w$ we get that 
$$
\frac{1}{2}\frac{\dd}{\dd t} \|w\|^2+ \|\nabla w\|^2=-\langle iP_3,w\rangle,\quad w(0)=0. 
$$
 By  \eqref{.3}, for $\omega\in\Omega_\delta$ the r.h.s. is bounded by 
 $C'(R,T)(\delta^2+ { \|w\|^2}+ \|w\|^4
 )$. Therefore 
  \begin{equation}\label{.5}
  \|w(T)\|^2\le e^{  { 2 C''(R,T)  } }\delta^2
 \end{equation}
 everywhere in $\Omega_\delta$, if $\delta$ is small.
  \smallskip
 
  \noindent
  {\it Step 5.}  Since $u=w+v_1+z$, then by 
     \eqref{.3}, \eqref{.1} and \eqref{.5}, for every $\delta, T>0$ and for each
  $\omega\in\Omega_\delta$ we have 
  $$
  \|u(T)\|\le \delta+e^{-\alpha_1T}R+e^{{ C''(R,T)T  }}{ \delta}=:\kappa.
  $$
  Since $u=v+z$, then 
  $|u(T)|_\infty\le \delta+ R+2$.  
 Choosing first $T\ge T(R,d)$ and next $\delta\le\delta(R,d,T)$ we achieve $\kappa\le d$. This proves the lemma
with $L= R+3 $. 
\end{proof}

\section{Ergodicity}\label{S:4}

In this section,  we analyse  behaviour of the process $u(t)$ with respect 
to the norms $\|u\|$ and $|u|_\infty$ 
and  next use an abstract theorem  from  \cite{KS11}
 to prove that the process   is mixing.
      
\subsection{Uniqueness of stationary measure
 and mixing}         
     First we recall the abstract theorem  from \cite{KS11} 
in the context of the CGL equation \eqref{E:1}.  Let us, as before, denote by  $P_t(u,\Gamma)$ and  $\PPPP^*_t$ 
the transition function and the
 family of Markov operators,  associated with  equation \eqref{E:1} in the space of Borel  measures in 
 $C_0(K)$. 
  Let $u(t)$ be a trajectory of \eqref{E:1}, starting from 
 a point $u\in C_0(K)$. Let $u'(t)$ be an independent copy of the process $u(t)$, starting from another point $u'$, 
 and defined on a probability space $\Omega'$ which is a copy of $\Omega$. For a closed subset  $B\subset C_0(K)$
 we set $\BBB=B\times B\subset C_0(K)\times C_0(K)$ and define 
 the hitting time 
 \begin{equation}\label{4.0}
  \tata(\BBB):=\inf\{t\ge0: u(t) \in B, u'(t)\in   B\},
  \end{equation}
  which is a random variable on $\Omega\times\Omega'$. 
  The following result is an immediate consequence of
    Theorem 3.1.3 in \cite{KS11}. 
  
  \begin{proposition}\label{T:hs}
Let us assume that for any integer $m\ge1$ there is a closed subset $B_m\subset {C_0(K)}$
and  constants $\de_m > 0,\ T_m\ge0$ such that $\de_m\ri 0$ as $m\ri \ty$, and the following  two 
properties hold:

(i) {\bf  (recurrence)}
For any $ u,u'  \in {C_0(K)} $,  
$\tata(\BBB_m)<\ty$ almost surely.

 (ii)  {\bf
 (stability)} For any $u,u'\in B_m$
  \begin{equation}\label{4.00}
  \sup_{t\ge T_m}\| P_t(u,\cdot)-P_t(u',\cdot)\|^*_{\LL({C_0(K)})} \le \de_m.  
  \end{equation}
Then the stationary measure $\mu$ of  eq. \eqref{E:1},
constructed in Theorem~\ref{T:BK},  is
unique  and for any $\la\in \PP({C_0(K)})$ we have
$\ 
\PPPP^*_t\la \strela \mu$ as  $t\ri \ty.$
  \end{proposition} 

We will derive from this that the Markov process, defined by eq.~\eqref{E:1} in $C_0(K)$, is mixing:
  
     \begin{theorem} \label{T:mix}
      There is an integer $N = N(B_*) \ge 1$ such that if 
  $ b_s\ne0$  for $|s|\le N$,  then   there is a unique stationary measure $\mu\in\PP(\CK)$ for (\ref{E:1}), and 
for any  measure $\lambda\in\PP(\CK)$ we have
$\ 
 \PPPP_t^*\lambda \strela  \mu$ as  $t\ri \ty.$
\end{theorem} 

The theorem  is proved in the next section. Now 
we derive from it a corollary:

  \begin{corollary}\label{C:43}
Let $f(u)$ be a continuous functional on $C_0(K)$ such that $|f(u)|\le   C_fe^{c|u|^2_\infty}$ for
 $u\in\CK$, where $ c<c_*$ ($c_*>0$ is the constant in assertion (i) of Theorem~\ref{P:1}). Then for
any solution $u(t)$ of (\ref{E:1}) such that 
    $u(0)\in C_0(K)$ is non-random, we have
    $$
    \E f(u(t))\to (\mu,f) \quad {\text as} \quad  t\to\infty.
    $$
        \end{corollary} 

\begin{proof}
For any $N\ge1$ consider a smooth function $\vp_N(r), 0\le\vp_N\le1$, such that  $\vp_N=1$ 
for $|r|\le N$ and $\vp_N=0$ for $|r|\ge N+1$.  Denote $f_N(u)=\vp_N(|u|_\infty)f(u)$.  Then $f_N\in C_b (\CK)$, 
so by Theorem~\ref{T:mix} we have
$$
|\E f_N(u(t))-(\mu,f_N) |\le \kappa(N,t),
$$
where $\kappa\to0$ as $t\to\infty$, for any $N$. 
  Denote $\nu^t(dr)=\DD(|u(t)|_\infty)$, $t\ge0$. Due to \eqref{E:6**}, 
\begin{equation*}\begin{split}
|\E(f_N(u(t))-f(u(t))|& \le C_f\int_0^\infty(1-\vp_N(r))e^{cr^2} \,\nu^t(dr)\\
&\le C_f e^{(c-c_*)N^2} \int_0^\infty e^{c_*r^2}\,\nu^t(dr)
\le C_1e^{(c-c_*)N}
\end{split}
\end{equation*}
(note that the r.h.s. goes to 0 when $N$ grows to infinity).  Similar, using Theorem~\ref{T:BK} we find that 
$\ 
|(\mu,f_N)-(\mu,f)|\to0$ as    $N\to\infty. 
$
The established relations  imply the claimed convergence.
\end{proof}

\subsection{Proof of Theorem~\ref{T:mix}} \label{s42}
It remains  to check that eq.~(\ref{E:1}) 
satisfies properties (i) and (ii) in Proposition~\ref{T:hs} for suitable sets $B_m$. 
For $m\in\N$ and $L>0$   we define 
$$
B_{m,L}:=\{u\in \CK: \|u \|\le \frac{1}{m}, |u |_\ty\le L\}
$$
(these are closed subsets of $\CK$).  For  $u_0, u'_0\in B_{m,L}$ consider solutions 
$$
u=u(t,u_0), \qquad u'=u(t, u'_0),
$$
defined on two independent copies $\Omega, \Omega'$ 
 of the probability space $\Omega$,
and consider   the first hitting time  $\tata(\BBB_{m,L})$ of the set $\BBB_{m,L}$ by 
 the pair  $(u(t), u'(t))$ (this is  a random variable on $\Omega\times\Omega'$, 
 see \eqref{4.0}). 
The proof of the following lemma is identical to that of Lemma~\ref{L:gn}.

 \begin{lemma}\label{L:gnl*}
There is a constant $L'>0$ such that for any $m\in\N$ we have
 \begin{align*} 
\E e^{\gamma \tata(\BBB_{m,L'})}\le C(1+ |  u_0|^2_\ty  + | u'_0|^2_\ty  )\,\,\,\text{for all
$u_0, u'_0\in \CK,$}  \end{align*}
 where $\gamma$ and $C$   are suitable positive constants.
\end{lemma}

Let us choose  $L=L'$ in the definition of the sets $B_{m,L}$ in Proposition~\ref{T:hs}. Then 
the property (i) holds and it remains to establish (ii), where  $P_t(u_0,\cdot)=\DD(u(t))$ and  $P_t(u'_0,\cdot)=\DD(u'(t))$. 
From now on we assume that the solutions $u$ and $u'$  are defined on the same probability space. 
It turns out that 
 it suffices to prove  \eqref{4.00} with the norm $\|\cdot\|^*_{\LL({C_0(K)})}$ replaced by  $\|\cdot\|^*_{\LL({H})}$. 
To show this we first  estimate the distance between $\DD(u(t))$ and $\DD(u'(t))$ in the Kantorovich metrics
\begin{align*}
\|\DD(u(t))-\DD(u'(t))\|_{K(H)} =\sup\{|(f,\DD(u(t)))-(f,\DD(u'(t)))| : \Lip(f)\le1\}
\end{align*}
in terms of  
$$d=\|\DD(u(t))-\DD(u'(t))\|_{\LL(H)}^*,$$
where $t\ge0$ is any fixed number. 
Without loss of generality, we can assume that   the supremum  in the definition of  the Kantorovich distance 
is taken over $f\in\LL(H)$ such that 
$\Lip(f)\le1$ and $f(0)=0$.
By \eqref{E:6**}, 
\begin{align}\label{y.2}
\E(e^{c\|u(t)\|}+e^{c\|u'(t)\| })\le C_L.
\end{align}
Setting $f_R(u)=\min\{f(u), R\}$ 
 and using \eqref{y.2}, the Cauchy--Schwarz and Chebyshev  inequalities, we get  
$$
\E|f(u(t))-f_R(u(t))|\le \E (\|u(t)\|-R)I_{\|u(t)\|\ge R}\le C_L' e^{-\frac{c}{2}R}.
$$
A similar inequality holds for $u'(t)$. Since $\|f_R\|_{\LL(H)}\le R+1$, then 
$$
\E|f(u(t))-f(u'(t))|\le 2C_L' e^{-\frac{c}{2}R}+(R+1)d.
$$
Optimising this relation in $R$, we find that 
$\ 
\E|f(u(t))-f(u'(t))|\le C_L'' \sqrt{d}.
$
Thus
$$\|\DD(u(t))-\DD(u'(t))\|_{K(H)} \le C_L'' \sqrt{d},
$$

By \eqref{3.33}, the functions $u(t)$ and $u'(t)$ belong to $C^\theta(K)$ for any $\theta\in (0,1)$. The following interpolation
inequality   is proved at the end of this section.
\begin{lemma}\label{L1}
For any   $u\in C^\theta(K)$ we have 
\begin{align}\label{y.4}
|u|_\ty\le C_{n,\theta}\|u\|^{\frac{2\theta}{n+2\theta}}|u|_{C^\theta}^{\frac{n}{n+2\theta}}.
\end{align} \end{lemma} 
\smallskip

By the celebrated Kantorovich theorem (e.g. see 
 in \cite{Dud}), we can find random variables $\xi$ and  $\xi'$ such that
  $\DD(\xi)=\DD(u(t))$,  $\DD(\xi')=\DD(u'(t))$ and
$$
\E\|\xi-\xi'\|=\|\DD(u(t))-\DD(u'(t))\|_{K(H)}\le  C_L'' \sqrt{d}.
$$
Using \eqref{y.4}, \eqref{3.33},  this  estimate and the H\"older 
inequality, we find that
\begin{align*}
\E|\xi-\xi'|_\ty\le C \E \|\xi-\xi'\|^{\frac{2\theta}{n+2\theta}} |\xi-\xi'|_{C^\theta}^{\frac{n}{n+2\theta}}\le
 ( C_L'' \sqrt{d})^{\frac{2\theta}{n+2\theta}} {C_L'''}^{\frac{n}{n+2\theta}}= \tilde C_L  {d} ^{\frac{\theta}{n+2\theta}} .
\end{align*}
Therefore, for any $f$ such that $\|f\|_{\LL(\CK)}\le 1$ 
 we have
 \begin{align*}
|(f,\DD(u(t)))&-(f,\DD(u'(t)))|=|\E f(\xi)-f(\xi') |\le \E | \xi - \xi'  |_\ty \le \tilde C_L  {d} ^{\frac{\theta}{n+2\theta}} ,
\end{align*}which implies that
 \begin{align}\label{y.5}
\|\DD(u(t)) -\DD(u'(t) )\|_{\LL(\CK)}^*  \le \tilde C_L \left(\|\DD(u(t))-\DD(u'(t))\|_{\LL(H)}^*\right)^{\frac{\theta}{n+2\theta}} .
\end{align} Thus we have proved 
 \begin{lemma} Assume that 
\begin{equation}\label{4.1}
  \sup_{t\ge T_m}\| P_t(u_0,\cdot)-P_t(u'_0,\cdot)\|^*_{\LL({H})} \le \de_m 
  \end{equation}
  for all $u_0,u'_0\in B_{m,L}$, where $\delta_m\to0$.  Then \eqref{4.00} holds for $B_m=B_{m,L}$ with
    $\delta_m' =C_L \de_m^{\frac{\theta}{n+2\theta}}$.
\end{lemma}

So to prove Theorem~\ref{T:mix} it remains to verify \eqref{4.1}.

 \begin{proof} [Proof of (\ref{4.1})]
 In view of  the triangle inequality we may assume that in \eqref{4.1}  $u_0'=0$.

 {\it Step 1.} In this step we prove that it suffices to establish \eqref{4.1} for solutions of an
 equation, obtained by   truncating the nonlinearity in  \eqref{E:1}.
 For any 
    $\rho\ge0$ and any continuous process $\{z(t):t\ge0\}$   with range in $\CK$ we define   the stopping time
    $$
    \tau^z= \inf \big\{ {t\ge0}:    \int_0^t|z(\tau)|^2_\infty  d\tau-Kt \ge\rho \},
    $$
    where $K$ is the constant in Lemma~\ref{P:3} 
   (as usual, $\inf\emptyset=\infty$).  We set  $\Omega_{\rho}^z=\{\tau^z<\infty\}$ and $\pi^z=\pP(\Omega_{\rho}^z)$.
 Then
  \begin{align}\label{*1}
\pi^u\le Ce^{-\gamma\rho}, \quad 
\pi^{u'}
\le Ce^{-\gamma\rho}
\end{align}
for  suitable  $C, \gamma>0$ and for any $\rho>0$. 
Consider the following auxiliary equation:
\begin{equation}\label{E:y1}
\dot v-\Delta v+i |v|^2v +\la P_N(v-u) =\eta(t,x), \qquad
v(0)=0.
\end{equation}
Consider $\tau^v$ and define $\Omega_\rho^v$ and $\pi^v$ as above. 
Define the stopping time
$$
\tau=\min\{\tau^u, \tau^{u'}, \tau^v\}\le\infty,
$$
and define the continuous 
processes $\hat u(t), \hat u'(t)$ and $\hat v(t)$ as follows: for $t\le\tau$ they coincide with the processes 
$u,u'$ and $v$ respectively, while for $t\ge\tau$ they satisfy the heat equation
$$
\dot z-\Delta z=\eta. 
$$
 Due to  \eqref{*1}
 \begin{equation}\label{100}
\|\DD (u(t))-\DD (\hat u(t))\|^*_\LL+ \|\DD (u'(t))-\DD (\hat u'(t))\|^*_\LL \le4\pP\{\tau<\infty\} 
\le8 Ce^{-\gamma\rho}+4 \pi^v.
\end{equation}
So  to estimate the distance between  $\DD(u(t))$ and  $ \DD(u'(t))$ it suffices to estimate $\pi^v$
and  the distance between $\DD(\hat u(t))$ and $\DD(\hat u'(t))$.
\smallskip

 {\it Step 2.} 
   Let us first estimate the distance between $\DD(\hat u(t))$  and   $\DD(\hat v(t))$.
 Equations (\ref{E:1}) and (\ref{E:y1}) imply that for $t\le\tau$ the difference 
 $w=\hat v-\hat u$ satisfies 
\begin{align*}
\dot w-\Delta w+i \big(
 |\hat v|^2\hat v-   |\hat u|^2\hat u \big)+\la P_Nw =0, \qquad
w(0)=-u_0,
\end{align*}
where $\ |\lag  |\hat v|^2\hat v - |\hat u|^2\hat u , w \rag| \le C(|\hat u|_\ty^2+|\hat v|_\ty^2)\|w\|^2$ .
Taking the ${H}$-scalar product of the $w$-equation with $2w$, we get  that 
\begin{align}
\frac{\dd }{\dd t}\| w\|^2&+2\|\nabla w\|^2  +2\la \|P_N w\|^2\le 
 C(|\hat u|_\ty^2+|v|_\ty^2)\|w\|^2 ,\qquad t\le\tau.
 \label{E:arc}
\end{align}
 Since
$\ 
\|\nabla w\|^2\ge \alpha_N \|Q_N w\|^2,
$
where $Q_N=\,$id$\,-P_N$,  then 
\begin{align*} 
2\|\nabla w\|^2 +2\la \|P_N w\|^2\ge 2 \la_1 \|w\|^2,\qquad \la_1:=\min\{\alpha_N,\la\}.
\end{align*}
 Choosing $\la$ and $N$ so large that $\la_1-CK\ge1$
and applying to  (\ref{E:arc})  the Gronwall inequality, we obtain that
\begin{align*}
\| w\|^2&\le    \|u_0\|^2 \exp \left(- 2\la_1 t+C \int_0^t (|\hat u|_\ty^2+|\hat v|_\ty^2)\dd s\right)\nonumber 
\\&\le   \frac{1}{m^2}\exp\left(-2(\la_1 -CK)t +2 C\rho \right)
 \le   \frac{1}{m^2} \exp \left(-2t +2 C\rho \right),
\end{align*}
for $t\le\tau$.  Clearly for $t\ge\tau$  we have 
$
(d/dt)\|w\|^2\le  - 2\|w\|^2.
$
Therefore 
\begin{equation}\label{E:pk}
\| w\|^2\le   \frac{1}{m^2} \exp \left(- 2t +2 C\rho \right)\qquad \forall\,t\ge0 \quad\text{a.s.}
\end{equation}
So for  any $f\in \LL(H)$ such that $\|f\|_{\LL}\le 1$   we get
\begin{align*} 
|\E(f(\hat u(t))-f(\hat v(t)))|&\le 
 \left(\E 
 \|w\|^2\right)^{\frac{1}{2}}  \le 
  \frac{1}{m} e^ {  C\rho-t }=:d(m ,\rho,t) .
\end{align*}
 Thus 
\begin{align}\label{*6}
\| \DD(\hat u(t))-\DD(\hat v(t))\|_{\LL(H)} ^*\le  
d(m ,\rho,t).
\end{align}

 {\it Step 3.} 
To estimate the distance between $\DD(\hat v(t))$  and   $\DD(\hat u'(t))$
notice that,  without loss of generality, we can assume that the underlying
probability space $(\Omega, \FF, \pP)$ is of the particular form: $\Omega$  is   the space of functions
 $u\in C(\R_+,\CK)$ that vanish at  $t = 0$, $\pP$  is the law of $\zeta$ defined by (\ref{E:2*}), and
$\FF$ is the completion of the Borel $\sigma$-algebra of $\Omega$ with respect to $\pP$. For any $\om_{\cdot}\in \Omega$,
 define the mapping $\Phi:\Omega\ri\Omega$ by
$$
\Phi(\omega)_t=\omega_t-\la\int_0^t  \chi_{s\le\tau} P_N\big(\hat v(s)-\hat u(s) \big)\dd s.
$$
 Clearly, a.s. we have
\begin{align}\label{E2}
 \hat u'^{\Phi(\omega)}(t)=\hat v^\omega (t)\quad   \text{ for all  $t\ge0$.}
\end{align}  
Note that the transformation $\Phi$ is finite  dimensional:
 it changes only the first $N$ components of a trajectory $\omega_t$. Due to  \eqref{E:pk},
 almost surely
 $$
 \int _0^\infty \|P_Nw(s)\|^2\,ds \le \frac{1}{2m^{2}} e^{2C\rho} .
 $$
 This relation,  the hypothesis that   
    $ b_s\ne0$ for any $|s|\le N,$ 
 and the argument in Section~3.3.3 of  \cite{KS11}, based on the Girsanov theorem,  show that
\begin{align}\label{*3}
\|\Phi  \circ \pP-\pP\|_{var}\le 
\frac{ C(\rho)}{m}  
 =:\tilde d(m ,\rho).
\end{align} 
Using (\ref{E2}), we get $\DD(\hat v(t))=\hat v_t\circ \pP=\hat u_t'\circ (\Phi\circ  \pP)$, where $\hat v_t$ stands for the random
variable  $\om\ri \hat v^\om(t)$. Therefore, 
\begin{align}\label{*4}
\|\DD (\hat v(t))-  \DD (\hat u'(t))\|_{\LL(H)}^*&\le2 \|\DD (\hat v(t))-  \DD (\hat u'(t))\|_{var}\nonumber\\
&\le 2\|\Phi  \circ \pP-\pP\|_{var}\le2 \tilde d(m ,\rho).
\end{align}

 {\it Step 4.} Now let us prove \eqref{4.1}. We get from \eqref{*1} and  \eqref{*3}   that 
\begin{align*} 
\pi^v=
\pP \Omega_\rho^{v}= \pP\Phi^{-1}(\Omega_\rho^{\hat u })=
(\Phi\circ\pP)\Omega_\rho^{\hat u }
\le \pP \Omega_\rho^{\hat u } + \tilde d(m ,\rho) \le
 Ce^{-\gamma\rho}  + \tilde d(m ,\rho).
\end{align*}

Due to \eqref{100}, \eqref{*6}, \eqref{*4} and the last inequality we have
\begin{equation*}\begin{split}
\|\DD(u(t))-\DD( u'(t))\|^*_\LL &\le 12Ce^{-\gamma\rho} + d(m,\rho,t)+6\tilde d(m,\rho)\\
&\le
12Ce^{-\gamma\rho} + \frac1{m}e^{C\rho-t}+ \frac6{m} C(\rho)=:D_m(t).
\end{split}
\end{equation*}
Let us  choose $\rho=\rho(m)$, where  $\rho(m)\to\infty$ in such a way that $\tfrac6{m} C(\rho(m))\to0$, 
and next take $T_m=C\rho(m)$. Then for $t\ge T_m$ we have $D_m(t)\le \delta_m\to0$.  This completes the proof.
   \end{proof}
   \begin{proof}[Proof of Lemma \ref{L1}]
 Let us take any $u\in C^\theta, u\not\equiv 0$ and set $M:=|u|_\ty$, $U:=|u|_{C^\theta}$. 
 Take any 
  $x_*\in K$ such that $|u(x_*)|=M$. To simplify the notation, we suppose that $x_*=0$. Regarding $u$ as an odd periodic function 
  on $\R^n$ we have 
   \begin{align*} 
 |u(x)|\ge M-|x|^\theta U\qquad \forall\,x. 
\end{align*}
The l.h.s of this inequality vanishes at $|x|=\left({M}/{U}\right)^{{1}/{\theta}}=:r_* \le1$. Integrating the squared relation 
 we get
 \begin{align*}
 \|u\|^2&\ge C\int_0^{r_*} (M-r^\theta U)^2r^{n-1}\dd r
 \\&=CU^2 \int_0^{r_*}(r_*^{2\theta}r^{n-1}-2r_*^\theta r^{n+\theta-1}+r^{n+2\theta -1})\dd r \\&= CU^2 r_*^{n+2\theta} (\frac{1}{n}-\frac{2}{n+\theta}+\frac{1}{n+2\theta})= U^2r_*^{n+2\theta} C(n,\theta)>0.
 \end{align*}
 Replacing in this inequality $r_*$  by its value we  get  \eqref{y.4}.
   \end{proof}

\section{Some generalisations} \label{S:5}

1) Our proof, as well as that of \cite{K99}, applies practically without any change to equations
\eqref{1}, where $\nu>0$ and $a\ge0$. Indeed, scaling the time and $u$ we achieve $\nu=1$
(the random force  scales to another force of the same type). Now consider equation \eqref{1}
with $\nu=1$ and $a\ge0$, and write the equation for $\xi(r(t,x))$. The integrand in the r.h.s.
of eq.~\eqref{equation} gets the extra term $-\xi'(r) a r^2$. Accordingly, the r.h.s. part $g(t,x)$ of 
eq.~\eqref{E:h1} gets the non-positive term $-ar^2$.  Since the proof in Section~\ref{S:2}  only
uses that $g\le \frac1{2r}\sum b_s^2| \varphi_s|^2$,  it does not change. In 
Sections~\ref{S:3}-\ref{S:4}, 
 as well as in \cite{K99}, we only use results of Section~\ref{S:2} and the fact that the 
nonlinearity in the equation, as well as its derivatives up to order $m$, admit polynomial 
bounds. For the argument in Section~\ref{S:4} it is important that the nonlinearity's derivative
grows no faster than $C|u|^2$. 

\smallskip

\noindent 
2) The proof of Theorem~\ref{T:1}, given in \cite{K99},  applies with minimal changes if the Sobolev space
$H^m(K)$ with $m>n/2$ (a Hilbert algebra) is replaced by the Sobolev  space $W^{1,p}(K)$ with $p>n$ (a
Banach algebra). It implies the assertions of the theorem with the norm $\|\cdot\|_m$ replaced by the norm
$|\cdot|_{W^{1,p}}$, under the condition that $B_1<\infty$. The argument in Sections~\ref{S:21}-\ref{s:32}
remains true in this setup since it does not use the $H^m$-norm. So to establish results of Section~\ref{S:3}
one can use the ${W^{1,p}}$-solutions instead of $H^m$-solutions. 
\smallskip

\noindent 
3)  Similar to 1) 
 results of  Sections~\ref{S:21}-\ref{s:32} remain true for eq.~\eqref{9}.
 \smallskip
 
\noindent
4) Consider equation \eqref{E:1} in a smooth bounded domain ${\cal O}\subset {\R}^n$ with Dirichlet boundary conditions:
\begin{equation}\label{5.1}
u\mid_{\partial \cal O}=0.
\end{equation}
Denote by $\{\vp_j, j\ge1\}$ the eigenbasis of $-\Delta$,
$$
-\Delta \vp_j=\lambda_j\vp_j,\quad j\ge1
$$
and define the random field $\zeta(t,x)$ as in Section~\ref{S:intr}, i.e.
 $\zeta=\sum_jb_j\beta_j(t)\vp_j(x)$. Denote
 $$
 B_*=\sum_j b_j|\vp_j|_\infty, \qquad B_1=\sum_jb_j^2|\nabla\vp|_p^2\,.
 $$
 The $W^{1,p}$-argument as in 2) applies to eq.   \eqref{E:1}, \eqref{5.1}  and proves an analogy of  Theorem~\ref{T:1}
 with the $\|\cdot\|_m$-norm replaced by the $|\cdot|_{W^{1,p}}$-norm, under the assumption that $B_*, B_1<\infty$. 
 The only difference is that now the assertion of Lemma~\ref{T:2*} follows not from \cite{K99}, but from the result of
 \cite{KNP03} (also see \cite{Kry, MR01}).

 After that the proof goes without any changes compared to Sections~\ref{S:intr}-\ref{S:4} and establishes for  equation  
 \eqref{E:1}, \eqref{5.1}  analogies of the main results of this work (with the space $C_0(K)$ replaced by $C_0(\cal O)$ 
 and $H^1$  by $H^1_0(\cal O)$):
 
   \begin{theorem} \label{T:final}
     Assume that  $B_*<\infty$. Then
     
     i) for any $u_0\in C_0(\cal O)$    problem  (\ref{E:1}),   (\ref{E:5}),  (\ref{5.1})  has a unique strong solution $u$ such that 
     $u\in {\cal H}(0,\infty)$ 
     a.s. This solution defines in the space $C_0(\cal O)$ a Fellerian  Markov process.
     
     ii) This process is mixing. 
   \end{theorem}

   The first assertion remains true if in eq. \eqref{E:1} we replace the nonlinearity by $ig_r(|u|^2)u$, $0<r<\infty$. If $r\le1$, then 
   the second assertion  is also true. It is unknown if the systems, corresponding to  equations with $r>1$, are mixing (this is a well known
   difficulty: it is unknown how to prove mixing for SPDEs without non-linear dissipation and with a conservative nonlinearity which 
   grows at infinity faster that in the cubic way).
    \smallskip

\noindent
5) Lemmas \ref{P:3}, \ref{L:gnl*} and estimate \eqref{y.5} allow to apply to eq.~\eqref{E:1} the methods, developed recently 
to prove exponential mixing for the stochastic 2d Navier-Stokes system (see in \cite{KS11} Theorems~3.1.7,~ 3.4.1 as well as
discussion of this result).  It implies that the Markov process, defined by eq.~\eqref{E:1}, is exponentially mixing, i.e. 
 in Theorem~\ref{T:mix}  the distance 
$\|\PPPP_t^*\lambda - \mu\|^*_\LL$ converges to zero exponentially fast. See Section~4 of \cite{KS11} for consequences of
this result.  Proof of this generalization is less straightforward than those in 1)-4) and will be presented elsewhere.


\bibliography{biblio}
\bibliographystyle{amsalpha}
\end{document}